\documentclass[oneside]{amsart}
\usepackage{amsmath,amsthm,amsfonts,amssymb,bm,wasysym}
\usepackage{amsxtra}
\usepackage{amstext}
\usepackage{hyperref}
\usepackage{verbatim}
\usepackage{float}
\usepackage{comment}
\usepackage[english]{babel}

\newcommand{\R}{\mathbb{R}}

\newcommand{\N}{\mathbb{N}}

\newcommand{\Z}{\mathbb{Z}}

\newcommand{\pp}{\mathbb{P}}

\newcommand{\kC}{\mathcal{C}}

\newcommand{\kO}{\mathcal{O}}

\newcommand{\kF}{\mathcal{F}}

\newcommand{\lin}{\left[\kern-0.15em\left[}
\newcommand{\rin} {\right]\kern-0.15em\right]}
\newcommand{\ilin}{\left]\kern-0.15em\left]}
\newcommand{\irin} {\right[\kern-0.15em\right[}

\newtheorem {lem} {Lemma} [section]

\newtheorem {theo} {Theorem} [section]

\newtheorem {conj} {Conjecture} [section]
\newtheorem {rem} {Remark} [section]

\def\epsilon{\varepsilon}
\def\tilde{\widetilde}
\newcommand{\1}{{\text{\Large $\mathfrak 1$}}}
\newcommand{\bs}{\boldsymbol}

\title[Localization on $5$ sites for VRRW]
      {Localization on $5$ sites for Vertex reinforced random walks: towards a characterization}

\author{Bruno Schapira}
\thanks{Aix-Marseille Universit\'e, CNRS, Centrale Marseille, I2M, UMR 7373, 13453 Marseille, France;  bruno.schapira@univ-amu.fr}

\begin{document}
\maketitle

\begin{abstract} We continue the investigation of the localization phenomenon for 
a Vertex Reinforced Random Walk on the integer lattice. 
We provide some partial results towards a full characterization of the weights for which localization on $5$ sites occurs with positive probability, and make some conjecture concerning the almost sure behavior. 
MSC 2010 \emph{subject classifications.} 60K35.
\end{abstract}

\section{Introduction} 
Given a sequence $w = (w(n))_{n\ge 0}$ of positive real numbers, called the weight, one can define a process $(X_n)_{n\ge 0}$ on $\Z$, called Vertex Reinforced Random Walk (VRRW) as follows: first $X_0=0$, and then for any $n\ge 0$ and $x\in \Z$, 
\begin{equation}\label{transiVRRW}
\pp(X_{n+1}= x\pm 1 \mid \kF_n)\, = \, \frac{w(Z_n(x\pm
1))}{w(Z_n(x+1)) + w(Z_n(x-1))},
\end{equation}
where $\kF_n:=\sigma(X_0,\ldots,X_n)$ and
$Z_n(y)$ is the number of visits to site $y$ by the process before time $n$ (see below). 
This process was introduced by Pemantle \cite{P} on the complete graph and for a linear weight, and then by 
Pemantle and Volkov on $\Z$, still for the linear weight, who showed that the process localizes on five 
sites with positive probability, that is with positive probability exactly five sites are visited infinitely often.  
This result was later improved by Tarr\`es who showed \cite{T1,T2} that this behavior occurs in fact almost surely.

A few years later, Volkov \cite{V} introduced the model with a general weight sequence, in the same fashion as Davis \cite{Dav} did for Edge Reinforced Random Walks. He proved in particular that for weights of the form $w(n)=n^\alpha$, with $\alpha<1$, 
localization on a finite subgraph is not possible. This was later improved in \cite{CK,Sch,S} in the case $\alpha<1/2$, where it was proved that 
the process visits almost surely all sites infinitely often.

In a previous work in collaboration with Basdevant and Singh \cite{BSS}, we managed to completely characterize the nondecreasing weights for which localization on $4$ sites occurs with positive probability, or almost surely, in terms of some parameter $\alpha_c(w)$ (see below). 
Our aim here is to analyze the analogous question for the  
localization on $5$ sites. For this we introduce 
some new parameter $\beta_c(w)$, which should play a similar role as $\alpha_c(w)$. To define it, we first extend $w$ as a function  
on the positive reals by $w(t): = w(\lfloor t\rfloor)$, and then set
\begin{equation*}
W(t) := \int_0^t \frac 1{w(u)}\, du.
\end{equation*}
We will assume throughout the paper that 
\begin{equation}\label{twosite}
\sum_{n=0}^{\infty}\frac{1}{w(n)} = \infty,
\end{equation}
which is equivalent to saying that 
$W$ is a bijection from $\R_+$ to itself. Note however, that this is not a restrictive hypothesis, since when $w$ is reciprocally summable, it is known \cite{BSS,V} that the process localizes almost surely on two sites. Then we denote by
$W^{-1}$ its inverse, and define for $\alpha>0$, 
\begin{equation*}
I_\alpha(w):= \int_0^\infty
\frac{dx}{w(W^{-1}(W(x)+\alpha))}.
\end{equation*}
When $w$ is nondecreasing, the map $\alpha
\mapsto I_\alpha(w)$ is nonincreasing and one defines
\begin{equation}\label{defAlphac}
\alpha_c(w):=\inf \{\alpha\ge 0 \ :\ I_\alpha(w)<\infty\}\in
[0,\infty], 
\end{equation}
with the convention that $\inf \emptyset = \infty$. In \cite{BSS} it was proved in particular that localization on $4$ sites holds with nonzero probability if, and only if, $\alpha_c(w)$ is finite. 
We now define for $\beta \in \R$, 
$$J_\beta(w):= \int_0^\infty \frac{dx} {w(W^{-1}(2W(x) + \beta))},$$
with the convention that $W^{-1}(u)=0$, for $u<0$, and set
$$\beta_c(w):= \inf\{ \beta\in \R\ :\ J_\beta(w)<\infty\}\, \in \, [-\infty,+\infty].$$
We make the following conjecture (with $R'$ standing for the set of sites which are visited infinitely often): 
\begin{conj} 
\label{conj}
Assume that $w$ is nondecreasing and satisfies \eqref{twosite}. Assume further that $\alpha_c(w)=\infty$. Then  
\begin{eqnarray*}
\pp(|R'|=5)> 0 \ \Longleftrightarrow \ \pp(|R'|=5)=1 \ \Longleftrightarrow \ \beta_c(w)<\infty.  
\end{eqnarray*} 
\end{conj} 
\begin{rem}\emph{
As we will later explain further, we also conjecture that in fact $\beta_c(w)$ always belongs to $\{\pm \infty\}$.}
\end{rem}
The hardest part here is the characterization of the almost sure localization, which is a notoriously difficult problem that we 
will not discuss in this paper; we simply recall that in the case of a linear weight, Tarr\`es proved that $|R'|=5$ almost surely \cite{T1,T2}. Proving that the same holds for some other weight function is possibly one of the most challenging problem on this model. Instead we will only be interested here on the easiest part of the conjecture, which is a characterization of the localization with positive probability. Our first result provides one direction of the conjecture:
\begin{theo}
\label{positif5}
Assume that $w$ is nondecreasing. Then 
$$\pp(|R'|=5)> 0 \quad \Longrightarrow \quad \beta_c(w) <+ \infty.$$ 
\end{theo}
We note that this result was proved in \cite{BSS2} (see the proof of Proposition 1.4 there) under some additional hypotheses on $w$, including the fact that $w$ was a slowly varying function.

Our second result concerns the other direction. However, instead of $\beta_c(w)$ being finite, one needs to assume some slightly stronger 
condition (which we nevertheless conjecture to be equivalent). Namely, we first define 
 $H(x):= x +W^{-1}(W(x)+1)$, and note that $H$ is increasing and continuous; thus it has an inverse which we denote by $H^{-1}$. Then set for $\beta \in \R$, 
$$\widetilde J_\beta(w) := \int_0^\infty \frac{dx} {w(H^{-1}\left(W^{-1}(2W(x) + \beta)\right))},$$ and 
$$\widetilde \beta_c(w):=\inf\{\beta\in \R\ :\ \widetilde J_\beta(w)<\infty\}\, \in\,  [-\infty,+\infty].$$
Note that $H(x)\ge x$ and $H^{-1}(x)\le x$, for all $x\ge 0$. 
Thus for any $\beta\in \R$, $\widetilde J_\beta(w)\ge J_\beta(w)$. In particular for any $w$, 
$$\widetilde \beta_c(w) \ \ge \ \beta_c(w).$$ 
Our second result is the following: 
\begin{theo}
\label{loc}
Assume that $w$ is nondecreasing and satisfies \eqref{twosite}. Assume further that $\alpha_c(w)=\infty$. Then 
\begin{eqnarray*}
\widetilde \beta_c(w) <\infty \ \Longrightarrow \ \left\{ 
\begin{array}{l} \pp(5\le |R'|<\infty)=1,  \\
                  \pp(|R'|\in \{5,6\})>0.
\end{array}
\right. 
\end{eqnarray*} 
\end{theo}
As mentioned above we conjecture that in fact $\beta_c(w) = \widetilde \beta_c(w)$, for all weights $w$. 
We provide some evidence for this fact at the end of the paper, and show that it is true for a large class of weight functions 
(see Lemmas \ref{surlinear} and \ref{sublinear}).

In particular Lemma \ref{surlinear} shows that for any surlinear weight function, such that $w(n)=o(n\sqrt{\log n})$, one has 
$\tilde \beta_c(w)=-\infty$. This is of course not surprising, regarding the known result for a linear weight, but we stress that prior to this, not much was known for weights with intermediate growth between linear and $n\log \log n$. Indeed, in \cite{BSS} it was only proved that for weights satisfying $w(n)=o(n\log \log n)$, $\alpha_c(w) = \infty$, and localization on $4$ or less sites was impossible.

It might look a bit disappointing that we cannot exclude the possibility of a localization on $6$ sites in the conclusion of Theorem \ref{loc}, especially since for a linear weight  
as well as for weights satisfying $w(n)\sim n/\exp(\log^\alpha n)$, with $\alpha\in (0,1/2)$, it was proved respectively in \cite{PV,T2} and \cite{BSS2}, that localization on $5$ sites occurs with positive probability. 
Let us however observe that in both cases the proofs rely heavily on the explicit form of the weight function and cannot be transposed (at least not directly)
to the general setting we are considering here.

Finally we also believe that localization on any even number of sites, larger than or equal to $6$, is not possible for any weight function. In contrast it was proved in \cite{BSS2} that localization on any odd number of sites -- other than one and three -- is possible.

The paper is organized as follows. In the next section, we recall some important and elementary facts about the VRRW, and some related martingales attached to each site. Then in Sections $3$ and $4$ we give the proofs of Theorems \ref{positif5} and \ref{loc} respectively. The final section is concerned with the computation of the parameters $\beta_c(w)$ and $\tilde \beta_c(w)$, and gives some cases where one can show equality between them. 

\section{Notation and background}

\subsection{VRRW}
Given some initial distribution of local times $\mathcal{C} :=(z_0(y))_{y\in \Z}\in \N^\Z$, we define the $\mathcal C$-VRRW as the process $(X_n)_{n\ge 0}$, 
whose transition probabilities are given by \eqref{transiVRRW}, with for any $y\in \Z$, $Z_0(y)=z_0(y)$, and for any $n\ge 1$, 
$$Z_n(y):=z_0(y)+\sum_{k=1}^n \1\{X_k=y\}.$$
We denote by $\pp_\kC$ the law of the $\mathcal C$-VRRW. 
We call $\kC_0$ the configuration with $z_0(y)=0$, for all $y\neq 0$ and $z_0(0)=1$. We then simply say that $X$ is a VRRW when its initial local time distribution is given by $\kC_0$, and denote its law by $\pp$. 
We also recall that a $\kC$-VRRW can be defined as well on any subgraph of $\Z$, and we refer to \cite{BSS} for details.  

\subsection{The martingales $M_n(x)$}\label{section3.2}
For $x\in \Z$, define $Z_\infty(x) := \lim_{n\to \infty} Z_n(x)$.
Recall that $R'$ stands for the set of sites visited infinitely
often by the walk:
$$R' :=\{x\in \Z\ :\ Z_\infty(x)=\infty\}.$$
We define for any $n\ge 1$, and $x\in \Z$, 
\begin{equation}\label{defY}
Y_n^{\pm}(x) := \sum_{k=0}^{n-1} \frac{\1\{X_k=x,\, 
X_{k+1}=x\pm 1\}}{w(Z_k(x\pm 1))},
\end{equation}
 and
$$M_n(x) := Y_n^+(x)-Y_n^-(x).$$
We let also $Y_0^\pm(x)=0$, and $M_0(x)=0$, and consider the limits: 
 $$
Y_\infty^\pm(x):= \lim_{n\to \infty} Y_n^{\pm}(x).
$$
An important observation from Tarr\`es \cite{T1,T2} is that
$(M_n(x))_{n\ge 1}$ is a martingale for each $x\in \Z$. Moreover, if
\begin{equation}\label{hyp.square}
\sum_{n=0}^{\infty} \frac{1}{w(n)^2}<\infty,
\end{equation}
then these martingales are bounded in $L^2$, and thus converge almost surely 
and in $L^2$. 
Moreover, for any $\mathcal{C}$-VRRW, one has 
\begin{equation}\label{eqW}
Y_n^+(x-1) + Y_n^-(x+1) = W(Z_{n}(x)) -W(z_0(x)). 
\end{equation}
\noindent We will also use the following result due to Tarr\`es (see also \cite[Lemma 3.3]{BSS}):  
\begin{lem}[Tarr\`es \cite{T2}] \label{Yfini} 
Assume that $w$ is nondecreasing and that \eqref{hyp.square} holds. Then, for any $x\in \Z$, almost surely,
$$\{Y_\infty^+(x)<\infty\}=\{Y_\infty^-(x)<\infty\}\ = \ \{Z_\infty(x-1)<\infty\} \cup \{Z_\infty(x+1)<\infty\}.$$
\end{lem}
We further use the same notation as in \cite{T2}, and write $f(n)\equiv g(n)$, when the sequence $(f(n)-g(n))_n$ converges to some finite real.
In particular, it follows from the above discussion that  
\begin{eqnarray}\label{Ypm.conv}
\sum_{n\ge 0} \frac{1}{w(n)^2}<\infty\quad \Longrightarrow\quad Y_n^+(x)\equiv Y_n^-(x), \text{ for all }x\in \Z.
\end{eqnarray}

\section{Proof of Theorem \ref{positif5}}
We start the proof with the following lemma: 
\begin{lem}\label{5.w.square}
Assume that $w$ is nondecreasing. Then 
$$\pp(|R'|=5)> 0 \quad \Longrightarrow \quad \sum_{n=0}^\infty \frac 1{w(n)^2} <\infty.$$
\end{lem}
\begin{rem}
\emph{This result has the same flavor than some others from \cite{Sch,S,V}, which all give different conditions on the weight $w$, ensuring that localization on any finite subgraph is not possible. In particular the proof in \cite{S} shows that for any weight satisfying $\limsup n/w(n)^2=\infty$, the walk cannot localize on any finite subgraph, which is close to imply our result (but not quite).}
\end{rem}
\begin{proof}[Proof of Lemma \ref{5.w.square}] We first note that if localization on five sites occurs with positive probability, then there exists some initial configuration $\kC$, such that with positive probability the $\kC$-VRRW spends all its time in the set $\{1,2,3,4,5\}$, and visits all sites from this set infinitely often. Call $E$ this event. By the conditional Borel-Cantelli Lemma (see Theorem 4.3.2 in \cite{Dur}), one can see  that almost surely on the event $E$, one has $Y_\infty^+(1)<\infty$, since for some constant $c>0$ (only depending on $\kC$), one has 
$$Y_\infty^+(1) \le c\sum_{k\ge 0} \pp\left[X_{k+1} = 0 \mid \kF_k\right] \1\{X_k=1\},$$
where we denote here by $\kF_k$ the sigma field generated by the process $X$ up to its $k$-th visit to site $1$. Then we use that the following process is a martingale (for a very similar reason as for $M_n(x)$):
$$Y_n^+(1) - \sum_{k=1}^{Z_n(2)-z_0(2)} \frac{p_k(2,1)}{w(k+z_0(2))},$$ 
where $p_k(2,1)$ denotes the probability to jump to site $1$ at $k$-th visit to site $2$. Since this martingale has bounded increments, we know that almost surely, either it converges, or its $\limsup$ as well as its $\liminf$ are both infinite (see Theorem 4.3.1 in \cite{Dur}). 
However, we have just observed that on the event $E$, its $\limsup$ is finite, which means that it must converge, and as a consequence on the event $E$, it holds almost surely 
$$ \sum_{k=1}^{\infty} \frac{p_k(2,1)}{w(k+z_0(2))} <\infty.$$
Now by definition of $p_k(2,1)$, one has for some constant $c>0$ (depending only on $\kC$), and on $E$,  
$$\sum_{k=1}^{\infty} \frac{1}{w(k+z_0(2))w(Z_{\tau_k}(3))}  \le c \sum_{k=1}^{\infty} \frac{p_k(2,1)}{w(k+z_0(2))}<\infty,$$
where $\tau_k$ denotes the time of $k$-th visit to site $2$.  
By symmetry one has as well  
$$\sum_{k=1}^{\infty} \frac{1}{w(k+z_0(4))w(Z_{\tilde \tau_k}(3))} <\infty,$$
with $\tilde \tau_k$ the time of $k$-th visit to site $4$. Finally observe that for any $n$, $Z_n(3) \le Z_n(2) + Z_n(4)+C$, with $C$ a constant depending only on $\kC$. This implies that for any $k$, either $Z_{\tau_k}(3)\le 2k+C+z_0(3)$  or $Z_{\tilde \tau_k}(3) \le 2k+C+z_0(3)$. Using that $w$ is nondecreasing, it follows that for some (possibly larger) constant $C>0$, 
$$\sum_{k\ge 0} \frac 1{w(2k+C)^2} <\infty.$$
The lemma follows, using again that $w$ is nondecreasing.    
\end{proof}
We next prove the following result. 
\begin{lem} \label{lem.Zn123}
Let $X$ be a $\kC$-VRRW, for some initial local time configuration $\kC$. Assume that $w$ is nondecreasing, and satisfies \eqref{twosite} and \eqref{hyp.square}. 
Then on the event $E=\{Z_\infty(0)=Z_\infty(4)=\infty\}\cap \{Y_\infty^+(0)<\infty\} \cap \{Y_\infty^-(4)<\infty\}$, it holds almost surely 
$$Z_n(2)-\max(Z_n(1),Z_n(3)) \to + \infty, \qquad \textrm{as }n\to \infty.$$
\end{lem}
\begin{proof}
Let for $n\ge 1$,
$$N_n(y,y\pm 1):=\sum_{k=0}^{n-1} \1\{X_k=y,\, X_{k+1}=y\pm 1\},$$ 
denotes the number of jumps from $y$ to $y\pm 1$ before time $n$, for any $y\in \Z$. 
Then 
\begin{equation}\label{ZNn}
Z_n(1)\equiv N_n(0,1) + N_n(2,1), \quad \text{and}\quad   
Z_n(2)\equiv N_n(1,2) +N_n(3,2).
\end{equation} 
Now observe that $W(N_n(2,1))-Y_n^-(2)$ is nondecreasing and that for all $n$, 
$$0\le W(N_n(2,1))-Y_n^-(2) \le Y_n^+(0)+W(z_0(1)).$$
Since by definition $Y_\infty^+(0)$ is finite on the event $E$, we deduce that  
$$W(N_n(2,1))\equiv Y_n^-(2) .$$
By symmetry, one has as well 
$$W(N_n(2,3))\equiv Y_n^+(2),$$
and since by \eqref{Ypm.conv}, one also has $Y_n^-(2) \equiv Y_n^+(2)$, we get in fact
\begin{eqnarray}
\label{W123}
W(N_n(2,1)) \equiv W(N_n(2,3)).
\end{eqnarray}
Moreover, Lemma \ref{Yfini} implies that under the hypotheses of the lemma and on the event $E$,  $Z_\infty(-1)$ is finite, and thus $Y_\infty^+(-1)$ also. 
Together with \eqref{eqW}, it follows that   
\begin{eqnarray}
 \label{W01}
W(N_n(1,0))\equiv W(Z_n(0))\equiv  Y_n^-(1)\equiv Y_n^+(1).
\end{eqnarray}
We claim now that   
\begin{eqnarray}
\label{WYn+}
\delta_n:= W(N_n(1,2)) - Y_n^+(1) \to +\infty.
\end{eqnarray}
Indeed, on one hand $\delta_n$ is nondecreasing, and on the other hand its limit $\delta_\infty$ satisfies $\delta_\infty \ge Y_\infty^-(3)$. 
Since $Z_\infty(4)=Z_\infty(2)=\infty$,  Lemma \ref{Yfini} shows that $Y_\infty^-(3)=\infty$, and we get \eqref{WYn+}. 
By using next that $|N_n(1,2)- N_n(2,1)|\le 1$, together with \eqref{W123}, \eqref{W01} and \eqref{WYn+}, we obtain
$$W(N_n(2,3))-W(N_n(1,0))\to +\infty,$$
which implies that $N_n(2,3)-N_n(1,0)\to +\infty$. Using now \eqref{ZNn}, it follows that $Z_n(2) - Z_n(1)\to +\infty$, almost surely. 
By symmetry we get as well $Z_n(2)-Z_n(3)\to + \infty$, and the lemma follows.  \end{proof}

\noindent Let us resume now the proof of Theorem \ref{positif5}. 
Lemma 3.7 in \cite{BSS} shows that there exists some local time configuration $\kC$, such that for the $\kC$-VRRW, the event 
$$E:= \{Z_\infty(0)=Z_\infty(4)=\infty\}\cap \{Y_\infty^+(0)<\infty\} \cap \{Y_\infty^-(4)<\infty\},$$
has some positive probability. 
Moreover, we know by \eqref{eqW} that on $E$, 
$$ W(Z_n(1)) - W(Z_n(3))\equiv Y_n^-(2)-Y_n^+(2),$$
and using \eqref{Ypm.conv}, we deduce that $W(Z_n(1)) - W(Z_n(3))$ converges as $n\to \infty$, towards some $\alpha\in\R$. 
Furthermore, Lemma 4.8 in \cite{BSS} shows that almost surely $\alpha \neq 0$, and by symmetry we can assume without loss of generality that $\alpha>0$. 
In particular, this gives $Z_n(1)\ge Z_n(3)$, for $n$ large enough.  
Set now 
$$h_n(2):= \sum_{k=1}^{Z_n(2)-z_0(2)} \frac{2p_k-1}{w(k+z_0(2))},$$
where $p_k$ is the probability to jump to site $1$ at $k$-th visit to site $2$. As noticed already in the proof of Lemma \ref{5.w.square}, one has 
$h_n(2) \equiv Y_n^+(1)-Y_n^-(3)$. 
But since after some time the process has at least probability $1/2$ to jump to $1$ 
when it is in $2$, we see that for $n$ large enough $h_n(2)$ is nondecreasing. In particular there exists some (random) constant $\gamma\in \R$, such that 
$h_n(2)\ge \gamma$, for all $n\ge 0$. This implies that for some other constant $\gamma'>0$, 
$$Y_n^-(3) \le Y_n^+(1) + \gamma',\quad \textrm{for all }n\ge 0.$$
By using also that 
$$W(Z_n(0)) \equiv Y_n^-(1) \equiv Y_n^+(1) \equiv W(Z_n(2)) -Y_n^-(3),$$
we deduce that for some (random) $\beta\in \R$, 
\begin{eqnarray}
\label{beta}
W(Z_n(2)) \le 2W(Z_n(0))+\beta,\quad \textrm{for all }n\ge 0.
\end{eqnarray}
Together with Lemma \ref{lem.Zn123}, this yields for some constant $c>0$,  
\begin{eqnarray*}
Y_\infty^+(0) \equiv  \sum_{n=0}^\infty \frac {\1\{X_n=0\}}{w(Z_n(1))} \ge c \sum_{n=0}^\infty \frac {\1\{X_n=0\} } {w(Z_n(2))} \ge  c \sum_{n=0}^\infty  \frac 1 {w(W^{-1}(2W(n) + \beta))},
\end{eqnarray*}
which concludes the proof of the theorem, since $Y_\infty^+(0)$ is finite on $E$. \hfill $\square$

\section{Proof of Theorem \ref{loc}}
We start the proof with some elementary lemma.  
\begin{lem} \label{lem.hypsquare}
Assume that $w$ is nondecreasing. Then 
$$\beta_c(w) <\infty \quad \Longrightarrow \quad \sum_{n=0}^\infty \frac 1{w(n)^2}<\infty.$$
\end{lem} 
\begin{proof} Assume that $w(n) \le \sqrt n$, for some $n\ge 1$. Since $w$ is nondecreasing, this implies on one hand $W(n+1)\ge \sqrt n$, and also $w(k)\ge w(0)$, for all $k\ge 0$. The latter implies the existence of a constant $c>0$, such that 
$W(k)\le \sqrt n/3$, for all $k\le c \sqrt n$ (namely one can take $c=w(0)/3$). Assume that $n$ is large enough so that $\sqrt n/3 < (\sqrt n - \beta_c(w)-1)/2$. 
Then $2W(k)+\beta_c(w)+1 <  \sqrt n$, for all $k\le c\sqrt n$. Therefore $W^{-1}(2W(k)+\beta_c(w)+1)< n+1$, for all such $k$, and  
it follows that 
$$\sum_{ c\sqrt n/2 \le k\le c\sqrt n} \frac 1 {w(W^{-1}(2W(k)+\beta_c(w)+1))} \ge c/2.$$
In particular, by definition of $\beta_c(w)$, this can only happen for finitely many $n$, which proves that $\liminf w(n)/\sqrt n \ge 1$.

We use now this information to bootstrap the previous argument. Assume that $w(n)\le n^{2/3}$ for some $n\ge 1$, later taken large enough. Note first that this implies $W(n+1)>n^{1/3}$. Moreover, since $\liminf w(m)/\sqrt{m}\ge 1$, we can find $c>0$ some small enough constant, such that $W(cn^{2/3}) \le n^{1/3}/3$, assuming $n$ is large enough. Taking larger $n$ if necessary, one can assume that $2W(k) +\beta_c(w) + 1 \le n^{1/3}$, for all $k\le cn^{2/3}$. This implies first $W^{-1}(2W(k)+\beta_c(w) +1) \le n$, and then $w(W^{-1}(2W(k)+\beta_c(w)+1))\le n^{2/3}$, for all such $k$. Thus 
$$\sum_{ cn^{2/3}/2\le k\le cn^{2/3}} \frac 1 {w(W^{-1}(2W(k)+\beta_c(w)+1))} \ge c/2,$$
from which we deduce that $\liminf w(n)/n^{2/3}\ge 1$, and the lemma follows. 
\end{proof}

\noindent The next step is the following lemma.   
\begin{lem} \label{stop}
Assume that $w$ is nondecreasing and that $\widetilde \beta_c(w)<\infty$. 
For $N\ge 1$ integer, $\eta \in (0,1)$, and $\beta\in \R$ we define ${\bs \kC}_{N,\eta,\beta}$, as the set
$${\bs \kC}_{N,\eta,\beta} := \left\{ \{z_0(x)\}_{x\in \Z} \in \N^\Z :  \begin{array}{l} 
z_0(-1)\le z_0(-2)+z_0(0) \\
z_0(-1)\wedge z_0(0)\ge N \\
W(z_0(-2))\le W(z_0(0))-\eta\\
W(z_0(-3)) \le W(z_0(-1))/2 - \beta
\end{array} 
 \right\}. $$ 
Given $\kC$ some local time configuration, we denote by $\pp_\kC^*$ the law of the $\kC$-VRRW  
restricted to the set $\{-3,\dots,0\}$.   
For any $\eta\in (0,1)$,  and $\beta > (\widetilde \beta_c(w)+3\eta)/2$, one has   
$$\lim_{N\to \infty} \inf_{\mathcal C\in \bs \kC_{N,\eta,\beta}}\pp_\kC^*\left(Y_\infty^+(-3) <\infty \right) = 1.$$ 
\end{lem}
\begin{proof} 
Let $N\ge 1$, $\eta>0$, and $\beta>(\widetilde \beta_c(w)+3\epsilon)/2$ be given. Consider $\kC\in \kC_{N,\varepsilon,\beta}$, and define the following stopping times: 
$$T_0:= \inf\{n\ge 0: W(Z_n(-2)) \ge W(N)-\eta\}, $$
$$T_1:=\inf\{n\ge T_0 : W(Z_n(-2))\ge W( Z_n(0))-\eta/2\},$$
\begin{eqnarray*}
T_2:=\inf\left\{n\ge T_1 : 
\begin{array}{l} 
W( Z_n(-2))\ge W( Z_n(0))-\eta/4 \\
\textrm{or}\\
W( Z_n(-2))\le W(Z_n(0))-3\eta/4 
\end{array}
\right\},
\end{eqnarray*}
and 
$$T_3 : = \inf\{n\ge 0  : W( Z_n(-3))\ge W(N)/2-\beta\}.$$
The main steps of the proof are the following. First we will see that on the event when $T_1$ is infinite, $Y_\infty^+(-3)$ is finite, 
and thus the main part of the proof is to deal with the event when $T_1$ is finite.  
Now after time $T_1$ we know that the local time in $-2$ is large, 
specifically 
$Z_{T_1}(-2) \ge N':=W^{-1}(W(N)-1)$, since $\eta\le 1$ by hypothesis. 
This ensures that the fluctuations of the martingale $(M_n(-1))_{n\ge 0}$ after time $T_1$ are small, by Doob's $L^2$ inequality combined 
with the fact that the square of $w$ is reciprocally summable by hypothesis on $\tilde \beta_c(w)$ and Lemma \ref{lem.hypsquare} (recall that $\beta_c(w)\le \tilde \beta_c(w)$). More precisely, we fix now some $\varepsilon>0$, 
and we get that for $N$ large enough, 
\begin{eqnarray}
\label{Mnx+2}
\pp_\kC^*\left(\sup_{n\ge T_1}\, |M_n(-1) - M_{T_1}(-1)|\ge \frac{\eta}{10} \right) \le 
\frac{200}{\eta^2} \sum_{i= N'}^\infty \frac 1{w(i)^2} \le \epsilon.  
\end{eqnarray}
The next step is to see that necessarily at time $T_1$ the local time in $-3$ is also large, and thus that 
the fluctuations of the martingale $(M_n(-2))_{n\ge 0}$ after $T_1$ are also small for $N$ large enough; see \eqref{Mnx+1} below. 
This is where the role of $T_0$ comes into play, since we show that the increment of the local time in $-3$ between $T_0$ and $T_1$ 
goes to infinity as $N\to \infty$. We note that to prove this, we use \eqref{Mnx+2}.  
The last step of the proof is to see that with probability close to one $T_2$ is infinite, and furthermore that the increment of $Y_n^+(-3)$ between times $T_1$ and $T_2$ is dominated by a series which appears in the definition of $\tilde \beta_c(w)$, which is why we need this quantity to be finite. This is also where the role of $T_3$ appears. We show that the process 
$$U(n) : = W( Z_n(-3)) - \frac {W( Z_n(-1))}2,$$ 
remains bounded between times $T_1\wedge T_3$ and $T_2$, with probability close to one (and is also bounded up to time $T_3$ by definition).

Let us now proceed with the details of the argument. First observe that $W(Z_n(-2))-W(Z_n(0)) \equiv  Y_n^+(-3)$, and thus  
\begin{eqnarray}
\label{Tfini}
\{T_1=\infty\} \subseteq \{ Y_\infty^+(-3)<\infty\}.
\end{eqnarray}
Therefore, one can assume now that $T_1$ is finite.

The next step is to show that the local time in $-3$ at time $T_1$ is large, and for this we show that its increment between $T_0$ and $T_1$ is large. 
Indeed, note first that since the process we consider is reflected in $0$, one has for any $n\ge T_0$, 
$$Y_n^+(-1)-Y_{T_0}^+(-1) = W( Z_n(0))-W( Z_{T_0}(0)).$$
In addition, \eqref{eqW} gives 
$$\left(Y_n^+(-3)-  Y_{T_0}^+(-3)\right) +\left( Y_n^-(-1) - Y_{T_0}^-(-1)\right) = W(Z_n(-2)) -W(Z_{T_0}(-2)).$$  
Assume that $N'=W^{-1}(W(n)-1)$ is large enough so that $w(N')\ge 6/\eta$. By definition of $T_0$ and $T_1$, this implies 
$$W( Z_{T_1}(-2)) - W( Z_{T_0}(-2)) \ge W( Z_{T_1}(0))- W( Z_{T_0}(0)) + \frac{\eta}{3}.$$
Then it follows from the last displays and \eqref{Mnx+2} that for $N$ large enough, 
\begin{align}
\label{Y+x}
\pp_\kC^*\left( Y_{T_1}^+(-3) -  Y_{T_0}^+(-3) \le \frac{\eta}{6}\right)\le \pp_\kC^*\left(| M_{T_1}(-1) -  M_{T_0}(-1)|\ge \frac{\eta}{6} \right)\le \epsilon. 
\end{align}
Define next 
$$N'':=\inf\{n\ge N' : W(n)-W(N')\ge \frac{\eta}{6}\}.$$
Since $w$ is nondecreasing, and since we recall that by definition of $T_0$, one has $Z_{T_0}(-2)\ge N'$, it follows from \eqref{eqW} that  
\begin{equation}\label{ZNx}
\left\{Y_{T_1}^+(-3) -  Y_{T_0}^+(-3) \ge \frac{\eta}{6}\right\} \, \subseteq \, \left\{Z_{T_1}(-3)- Z_{T_0}(-3) \ge N''-N'\right\}.
\end{equation}
However, $N''-N'\to \infty$, and thus $Z_{T_1}(-3) \to \infty$ as well, when $N\to \infty$. It follows using again Doob's $L^2$-inequality, that for $N$ large enough, 
\begin{eqnarray}
\label{Mnx+1}
\pp_\kC^*\left(\sup_{n\ge T_1}\, | M_n(-2) -  M_{T_1}(-2)|\ge \frac{\eta}{6} \right) \le \epsilon.
\end{eqnarray}
The last step of the proof is to show that with probability close to one, $T_2$ is infinite and $Y_{T_2}^+(-3) - Y_{T_1}^+(-3)$ is  finite. 
Let 
$$h(n):=\sum_{k=1}^{ Z_n(-1)} \frac {2p_k-1}{w(k)},$$
where $p_k$ is the probability to jump to $0$ at $k$-th visit to $-1$. 
Recall that 
$$ Y_n^-(0)- Y_n^+(-2) \equiv h(n),$$
and on the other hand \eqref{eqW} and Lemma \ref{Yfini} yield 
$$W( Z_n(-1))\equiv  Y_n^+(-2) +  Y_n^-(0),\quad \text{and}\quad  Y_n^+(-2)\equiv  Y_n^-(-2)\equiv W( Z_n(-3)).$$
As a consequence,  
$$U(n)  = W( Z_n(-3)) - \frac {W( Z_n(-1))}2 \equiv -\frac {h(n)}2.$$
Since $Z_n(-2)\le Z_n(0)$, for all $n\le T_2$, 
$h$ is nondecreasing up to time $T_2$. 
Note that $U(T_1\wedge T_3) \le -\beta+\eta/2$, if $N$ is taken large enough. Also by definition, $\sup_{n\le T_3} U(n)\le -\beta$. 
Therefore by using \eqref{Mnx+1}, and again Doob's $L^2$-inequality, we get at least for $N$ large enough, 
\begin{eqnarray} 
\label{Yn+x+1}
\pp_\kC^* \left(\sup_{T_1\wedge T_3 \le n \le T_2} U(n) \ge -\beta+\eta \right)\leq \epsilon.
\end{eqnarray} 
Remember then that $H$ is defined by $H(x)=x+W^{-1}(W(x)+1)$, and thus by using the hypothesis on $\kC$, 
we get that for all $T_1\le n\le T_2$, 
$$Z_n(-1)\le Z_n(-2) + Z_n(0) \le H( Z_n(-2)).$$
It follows that on the event $\{\sup_{T_1\wedge T_3 \le n \le T_2} U(n)\le -\beta+\eta\}$, one has 
\begin{eqnarray}
\nonumber  Y_{T_2}^+(-3)- Y_{T_1}^+(-3) &=& \sum_{n=T_1}^{T_2} \frac {\1\{X_n=-3\}}{w(Z_n(-2))}
\le   \sum_{n=T_1}^{T_2} \frac {\1\{X_n=-3\}}{w(H^{-1}(Z_n(-1)))} \\
\label{ZTx} &\le & \sum_{n= Z_{T_1}(-3)}^{\infty} \frac 1 {w(H^{-1}(W^{-1}(2W(n) + 2\beta - 2\eta )))}. 
\end{eqnarray}
Using now that $\beta \ge (\widetilde \beta_c(w)+3\eta)/2$, we can find $K\ge 1$ such that 
$$\sum_{n=K}^{\infty} \frac 1 {w(H^{-1}(W^{-1}(2W(n) + \widetilde \beta_c(w) +\eta )))}\le \frac{\eta}{10}.$$
Then by using \eqref{Y+x}, \eqref{ZNx}, \eqref{Yn+x+1} and \eqref{ZTx}, we get that if $N$ is large enough, 
\begin{eqnarray}
\label{YT'T}
\pp_\kC^* \left( Y_{T_2}^+(-3)- Y_{T_1}^+(-3) > \frac{\eta}{10} \right) \le 2\epsilon.
\end{eqnarray}
But by definition of $T_1$ and $T_2$,  on the event $\{T_2<\infty\}$, we have for $N$ large enough, 
\begin{eqnarray*}
\eta/5 & < & \{ W( Z_{T_2}(-2)-W(Z_{T_1}(-2))\}-\{W( Z_{T_2}(0)-W(Z_{T_1}(0))\} \\
&= & \{Y_{T_2}^+(-3)- Y_{T_1}^+(-3)\} - \{ M_{T_2}(-1)- M_{T_1}(-1)\}.
\end{eqnarray*}
Therefore \eqref{Mnx+2} and \eqref{YT'T} imply  
\begin{eqnarray*}
\pp_\kC^*(T_2<\infty)\le 3\varepsilon, \quad \text{and}\quad  \pp_\kC^*\left( Y_\infty^+(-3)- Y_{T_1}^+(-3) \ge 1  \right)\le 5\epsilon. 
\end{eqnarray*}
Since $\epsilon>0$ can be chosen arbitrarily small, and since we recall that on the event $\{T_1=\infty\}$, one has $Y_\infty^+(-3)<\infty$, this concludes the proof of the lemma.  
\end{proof}
\noindent We can now finish the proof of Theorem \ref{loc}. Fix some $\eta\in (0,1)$ and $\beta>(\tilde \beta_c(w)+3\eta)/2$, and consider some initial local time configuration $\kC$,  such that $z_0(-1)\ge N$, $z_0(0)\ge N$, and $z_0(x)=0$, for $x\notin \{-1,0\}$, with $N\ge 1$.  
Note that by definition $\kC\in \bs \kC_{N,\eta,\beta}$, and thus by Lemma \ref{stop} one has  
 $\pp_\kC^*(Y_\infty^+(-3)<\infty)\ge 3/4$, for $N$ large enough. 
Using the continuous time-line construction of the VRRW (also called Rubin's construction, see  \cite{T2,BSS}), we can couple the process reflected in $-3$ and $0$, say $X$, with the process reflected in $-3$ and $2$, say $\tilde X$, and Lemma 3.6 in \cite{BSS} (see also \cite{T2} for a similar result) 
tells us that $\tilde Y_\infty^+(-3)\le Y_\infty^+(-3)$. Applying this argument twice, we see that for $N$ large enough, with probability at least $1/2$, one has both $\tilde Y_\infty^+(-3)$ and $\tilde Y_\infty^-(2)$ finite. Then using Lemma 3.7 in \cite{BSS}, we deduce   
that for $N$ large enough (say larger than some $N_0$) the unreflected $\kC$-VRRW on $\Z$ never visits sites $-4$ and $3$, with some positive probability. 

\vspace{0.2cm}
\noindent  
Then we can see that the same holds for the $\kC_0$-VRRW, since for any $N\ge 1$, 
with positive probability at time $2N$, we have $Z_N(-1)=N-1$, $Z_N(0) = N$, and $X_{2N}=0$, and one can then apply the previous result at time $2N$. 
This proves in particular that $|R'|\in \{5,6\}$, with positive probability. 

\vspace{0.2cm}
\noindent Now it just remains to show that almost surely the walk visits only a finite number of sites. However, each time the VRRW on $\Z$ visits a new site $x<0$, two cases may appear. If at this time the local time in $x+1$ is not larger than $N_0$, then the process 
has some positive probability (depending only on $N_0$) to jump immediately to $x-1$, and then to localize on the set $\{x-7,\dots,x-1\}$ and never come back to $x$, by the above argument. 
If instead at this hitting time of $x$, the local time in $x+1$ is larger than $N_0$, then necessarily the local time in $x+2$ has to be also not smaller than $N_0$, and we deduce by using again the above argument, that the process has some (constant) positive probability to never visit $x-2$. 
Then the conditional Borel-Cantelli lemma (see Theorem 4.3.2 in \cite{Dur}) shows that almost surely $\inf_n X_n>-\infty$. By symmetry we also get that almost surely $\sup_n X_n<\infty$, and this concludes the proof of Theorem \ref{loc}. \hfill $\square$

\section{On the values of the parameters $\beta_c(w)$ and $\widetilde \beta_c(w)$.}
Let us first observe that for any nondecreasing $w$, and any $\lambda>0$, one has $\beta_c(\lambda w) =  \beta_c(w)/\lambda$, and $\tilde \beta_c(\lambda w) =  \tilde \beta_c(w)/\lambda$ (which follows from the facts that $J_\beta(\lambda w) =  J_{\lambda \beta}(w)/\lambda$, and $\tilde J_\beta(\lambda w) = \tilde J_{\lambda \beta}(w)/\lambda$).

Now our aim here is to convince the reader that 
in most cases (and we believe this is true in fact for any nondecreasing weight function), one has:
\begin{eqnarray}
\label{conjpoids}
\beta_c(w)=\widetilde \beta_c(w)\in \{-\infty,\infty\}.
\end{eqnarray}
On one hand we prove in Lemma \ref{surlinear} that this is true for any weight function growing at least linearly and not faster than $n\sqrt{\log n}$.  
On the other hand, we show in Lemma \ref{sublinear} that it holds as well for a large class of sublinear weights.

Now recall that one can restrict our attention to weights satisfying \eqref{twosite} and such that $\alpha_c(w)=\infty$, since 
otherwise we already know the behavior of the process by the results of \cite{BSS}. But it is also proved there 
that if $\liminf w(n)/(n\log \log n)>0$, then $\alpha_c(w)$ is finite; thus the upper bound on $w$, which is imposed in the hypotheses of 
Lemma \ref{surlinear} below is not a strong restriction. 

\begin{lem} \label{surlinear} 
Let $w$ be some nondecreasing weight function, such that 
$$\liminf_{n\to \infty} \frac{w(n)}{n}>0.$$  
\begin{itemize}
\item[$\bullet$] If $w(n)=o(n\log n)$, then $\beta_c(w)=-\infty$. 
\item[$\bullet$] If $w(n) = o(n\sqrt{\log n})$, then $\beta_c(w)= \tilde \beta_c(w)= -\infty$. 
\end{itemize}
\end{lem}
\begin{proof}
Assume first that $w(n) = o(n\log n)$. Let $\epsilon>0$, be such that $w(n)\le \epsilon^2 n\log n$, and $w(n)\ge \epsilon n$, for all $n$ large enough. Then at least for $t$ large enough,  
\begin{equation}\label{W.eq1}
W(t)\ge \frac 1{2\epsilon^2} \log \log t. 
\end{equation}
Now by definition, for any $\beta \in \R$, and $t$ large enough,  
$$W(t)+\beta = \int_t^{W^{-1}(2W(t)+\beta)} \frac{du}{w(u)}.$$
Thus using that $w(n)\ge \epsilon n$, we get that for $t$ large enough,
\begin{equation*}\label{W.eq2}
W(t) +\beta \le \frac 2{\epsilon} \log \left(\frac {W^{-1}(2W(t)+\beta)}{t}\right).
\end{equation*}
Combining this with \eqref{W.eq1}, we get,
$$W^{-1}(2W(t)+\beta) \ge t(\log t)^{\frac 1{4\epsilon}}, $$
for all $t$ large enough. Then by choosing $\epsilon<1/8$, and using again that $\liminf w(n)/n>0$, the first assertion of the lemma follows.

Assume now that $w(n)=o(n\sqrt{\log n})$, so that for $n$ large enough $w(n)\le \epsilon^2 n\sqrt{\log n}$. 
Then for $x$ large enough, 
\begin{align*}
1&=\int_x^{W^{-1}(W(x)+1)} \frac{du}{w(u)}\ge \frac{1}{\epsilon^2} \int_x^{W^{-1}(W(x)+1)} \frac{du}{u\sqrt{\log u}}\\
& \ge \frac 2{\epsilon^2} \left(\sqrt{\log W^{-1}(W(x)+1)} - \sqrt{ \log x}\right).
\end{align*}
Thus for $x$ large enough (and $\epsilon$ small enough), 
$$H(x) = x+ W^{-1}(W(x)+1) \le 2 x \cdot e^{\epsilon^2\sqrt{\log x}}.$$ 
On the other hand, a similar argument as above shows that for any $\beta\in \R$, for $x$ large enough, 
$$W^{-1}(2W(x) + \beta)\ge x \exp(\sqrt{\log x}).$$
In particular, by taking $\epsilon$ small enough, we get that for $x$ large enough, 
$$H^{-1}(W^{-1}(2W(x)+\beta)) \ge x(\log x)^2, $$
and the second assertion of the lemma follows, using again that $\liminf w(n)/n$ is positive.  
\end{proof}

\noindent Our second result is concerned with sublinear weights.  
\begin{lem} 
\label{sublinear}
Let $w$ be some nondecreasing weight function satisfying \eqref{twosite}. 
If the two following conditions hold: 
\begin{eqnarray}
\label{Halpha}
\limsup \ W^{-1}(W(n)+\alpha)/n <\infty, \quad  \textrm{for any } \alpha >0, 
\end{eqnarray}
and
\begin{eqnarray}
\label{wc} 
\limsup \ w(cn)/w(n) <\infty, \quad \textrm{for any } c>1,
\end{eqnarray}
then $\beta_c(w) = \widetilde \beta_c(w) \in \{-\infty,\infty\}$. Moreover, if $w(n)=\kO(n)$, then \eqref{Halpha} holds. 
\end{lem} 
\begin{proof}
By using a change of variables, we can write for any $\beta>\beta'$, for some constant $c>0$, 
\begin{align*}
J_\beta(w)& =\int_0^\infty \frac {dt}{w(W^{-1}(2W(t)+\beta))} = \int_0^\infty \frac{w(W^{-1}(u))}{w(W^{-1}(2u+\beta))} \, du\\
& \ge c \int_0^\infty \frac{w(W^{-1}(u))}{w(W^{-1}(2u+\beta'))} \, du=cJ_{\beta'}(w),
\end{align*}
using the two hypotheses of the lemma. This implies that $\beta_c(w) \in \{\pm \infty\}$. Moreover, by \eqref{Halpha}, one has $t\le H(t)\le Ct$, 
for some constant $C>0$, and all $t$. It follows using \eqref{wc} that $\widetilde J_\beta(w) \le CJ_\beta(w)$, for some possibly larger $C$, and all $\beta$, and 
we deduce that $\widetilde \beta_c(w) \in \{\pm \infty\}$ as well.

\noindent Now if there exists $C>0$, such that $w(n) \le Cn$, for all $n\ge 1$, then for any $\alpha>0$,  
$$\alpha = \int_n^{W^{-1}(W(n)+\alpha)} \frac{dt}{w(t)} \ge \frac 1 C \int_n^{W^{-1}(W(n)+\alpha)} \frac{dt}{t} =\frac 1C \log \left(\frac{W^{-1}(W(n)+\alpha)}{n}\right),$$
which proves the second assertion of the lemma. 
\end{proof}

Let us conclude this section by mentioning that by combining the results of \cite{BSS2} with our Theorem \ref{positif5}, and the previous lemma, 
we obtain that any nondecreasing weight function $w$, such that $w(n)\sim n\exp(-(\log n)^\alpha)$, for some $\alpha \in (0,1/2)$ satisfies $\tilde \beta_c(w)=-\infty$. Indeed, we know from \cite{BSS2}, that for such weight localization on $5$ sites occurs with positive probability. Then Theorem  \ref{positif5} shows that $\beta_c(w)$ is finite, and finally Lemma \ref{sublinear} gives that in fact $\beta_c(w)=\tilde \beta_c(w)=-\infty$.  
On the other hand, when $w(n)\sim n\exp(-(\log n)^\alpha)$, with $\alpha>1/2$, the results of \cite{BSS2} show that $J_0(w)=\infty$, which imply $\beta_c(w)\ge 0$. Observing that $w(n) = \mathcal O(n)$, and applying Lemma \ref{sublinear} gives $\beta_c(w)=\tilde \beta_c(w) = \infty$. 

\bigskip 

\textbf{Acknowledgments:} I warmly thank an associate editor for his many comments, which helped correct some mistake and improve greatly the readability of the manuscript.

\end{document}